\documentclass[a4paper,12pt]{amsart}          
\usepackage{amsfonts,amsmath,latexsym,amssymb} 
\usepackage{amsthm}      
\usepackage[dvipsnames]{xcolor}          
\usepackage{float}
\usepackage{tikz}
\usepackage{circuitikz}
\usepackage{hyperref}
\usepackage{mathtools}
\usepackage{enumitem}
\usepackage[pdftex]{lscape}
\usepackage{geometry}
\allowdisplaybreaks

\newtheorem{theorem}{Theorem}[section]
\newtheorem{proposition}[theorem]{Proposition}

\newtheorem{corollary}[theorem]{Corollary}

\newtheorem{definition}[theorem]{Definition}
\newtheorem{example}[theorem]{Example}
\newtheorem{remark}[theorem]{Remark}

\author [Rolando Jimenez, 
V. Vershinin, Y. Muranov]{
 Rolando Jimenez, 
Vladimir Vershinin, Yuri Muranov}

\address{Rolando Jimenez: Instituto de Matematicas,
UNAM, Unidad Oaxaca, 
68000 Oaxaca, Mexico}

\email{rolando.jimenez@im.unam.mx}

\address{Yuri Muranov:
Faculty of Mathematics  and Computer Science, 
University of Warmia and Mazury
 in Olsztyn, ul. Sloneczna 54, 10-710 Olsztyn, Poland}

\email{muranov@matman.uwm.edu.pl}

\address{Vladimir Vershinin:
D\'epartement des Sciences Math\'ematiques,
Universit\'e de Montpellier,  Place Eug\'ene Bataillon,
34095 Montpellier
cedex 5,
 France}

\email{vladimir.vershinine@umontpellier.fr}

\title[On Cubical Sets]{On  Cubical Sets of Quivers and Digraphs}
\thanks{}

\begin{document}
\maketitle

\begin{abstract}
 The singular cubical homology theory for the category of quivers or digraphs  can be 
 constructed similarly to the classical singular homology theory for topological spaces. 
 The case of digraphs and quivers  differs from the topological case  due to the possibility of using a large 
 number of non-isomorphic line digraphs that  correspond to the unit interval in 
 algebraic topology. In this paper   we introduce several different notions of quiver  realizations of a 
cubical set  and we describe  relations between  them. We also define  various singular cubical homology theories  on 
quivers and  digraphs. Moreover,  using  quiver realizations of  cubical sets we define  a collection of 
path homology theories on the category of cubical sets and we describe their properties.

\end{abstract}

{\footnotesize{\bf Keywords:} {box category, cubical set, cubical set of a quiver, homology of digraph, homology 
of quiver,  singular cubical homology,      topological realization of  cubical set, digraph realization of  
cubical set, quiver realization of cubical set, path homology, homology  of cubical set}
\smallskip

}

\tableofcontents

\section{Introduction}
\setcounter{equation}{0}
\label{S1}

The singular cubical homology  of a topological space $X$ was defined using the singular    
$n$-cubes  which are given by    continuous 
maps $f\colon I^n\to X$ where $I^n$ is the $n$-dimensional  cube for the unit segment $I=[0,1]$  
 \cite{Serre, Hatcher, Hilton}. For the category of digraphs, the singular cubical homology theory was constructed 
in \cite{Mi4} using the  line digraph 
$(0\to 1)$ instead of the  segment $I=[0,1]$   in the topological case. 
Afterwards  there was constructed  a collection of \emph{cell singular  homology
theories of digraphs}   for $n\geq 1$ using  line digraphs 
$
(0\to 1\to\dots \to n)$ instead of the digraph  $(0\to 1)$ \cite{Mi}.  Thus, for every $n\geq 1$, we have a singular 
cubical homology theory  which   is  based  on the  digraph
$0\to 1\to\dots \to n$ which   corresponds to the unit interval $I$ in algebraic topology.  

The notion of a cubical set was introduced by Kan as an  algebraic  model for the investigation of singular cubical complex 
$S^{\Box}(X)$ of a topological space $X$~\cite{Kan}.   A cubical set  is a discrete object which is based  on an 
union of cubes in various dimensions with a collection of  special relations. This set is a natural  analog of the  
simplicial set which is based on an union of simplexes (see, also,  \cite{Antolini}, \cite{Nonab}, \cite{Grandis1}).  
Every cubical set $K$ admits a topological realization $|K|_{\bold{Top}}$ which is a $CW$-complex. Moreover,  
for every topological space $X$ there is a week  homotopy equivalence 
$\left|S^{\Box}(X)\right|_{\bold{Top}}\sim X$~\cite{May, Nonab}.

In the present paper,  we introduce various  notions of a quiver  realization of a cubical set    and    various singular 
cubical sets of quivers and digraphs, and we  describe relations between introduces objects. Using this approach we 
construct several singular and path homology theories on the categories of cubical sets,  quivers and digraphs.

In Section \ref{S2},  we give preliminary information and necessary definitions.

In Section \ref{S3}, we introduce a category of quivers $\mathbf Q$ which contains the subcategory $\mathbf D$ 
of digraphs.  Let  $I_k$ be a line digraph consisting of $(k+1)$ vertices 
$0,1, \dots ,k$ and  containing only one  arrow  $i\to i+1$ or $i+1\to i$ between any pair of 
consequent vertices and containing  no other arrows. We define a digraph $n$-cube as the Box product of $n$ 
digraphs $I_k$ and  the box category  $\Bbb D_{I_k}$ using the cubes $I_k^n$  and the morphisms are obtained as
compositions of face inclusions and projections on faces. We introduce the notion of  
homotopy in the category $\mathbf Q$. 

In Section \ref{S4},  for every  quiver $Q$  and a line digraph $I_k$,  we define a singular cubical set 
$S^{\Box}_{I_k}(Q)$  which is given by singular cubes  $\varphi\colon I_k^n\to Q$ and describe relations 
between such sets for different line digraphs
$I_k$.  The consideration of
such singular cubes in digraphs gives singular cubical sets in the category $\mathbf D$. 
 For every cubical set $K$  and a line digraph $I_k$ we define a quiver realization 
$\left|K\right|_{\bold Q}^{I_k}$ and describe relations between the functors of quiver realization of a cubical 
set and  the singular cubical functors of quivers. 

In Section \ref{S5}, we apply introduced constructions to define several homology  theories of cubical sets and quivers 
and we describe relations between them. 



\section{Preliminaries}
\label{S2} \setcounter{equation}{0}

Recall the notions of the \emph{box category} and \emph{cubical set}
 \cite{Nonab, Grandis1, Kan}.  For $n\geq 1$ let $\Bbb I^n$   be the $n$-fold product of copies  of the set $\Bbb I=\{0,1\}$, and  let $\Bbb I^0=\{0\}$ be the one-element set. 

\begin{definition}\label{d2.1} \rm The \emph{box category} $\Box$ has objects $\Bbb I^n$ for $n\geq 0$ and  morphisms of this category can be obtained by composition of  \emph{face inclusions}  
$$
\delta_i^{\alpha}\colon \Bbb I^{n-1}\to \Bbb I^{n}  \ \text{given by}\  \delta_i^{\alpha}(x_1, \dots , x_{n-1})= (x_1, \dots ,x_{i-1}, \alpha, x_{i}, \dots x_{n-1}), \ \alpha\in \Bbb I 
$$
where $(1\leq i\leq n)\& (n\geq 2)$ and $\delta_1^{\alpha}(0) = \alpha$ for $n=1$ 
and \emph{projections}
$$
\sigma_i\colon \Bbb I^{n}\to \Bbb I^{n-1}  \ \text{given by}\  \sigma_i(x_1, \dots , x_{n})= (x_1, \dots ,x_{i-1}, x_{i+1}, \dots x_{n}) 
$$
where $n\geq 2$,  $1\leq i\leq n$,   and $\sigma_1(x_1)=0$ for $n=0$.
 \end{definition}

\begin{proposition}\label{p2.2} Any morphism in $\Box$  has a unique
expression as a composition 
$$
\delta_{i_1}^{\alpha_1}\dots \delta_{i_k}^{\alpha_k}\sigma_{j_1}\dots \sigma_{j_l}  \ \text{where} \  i_1\leq \dots \leq i_k, \ j_1\leq \dots \leq j_l.
$$
There are the following relations between morphisms  of the category $\Box$:
$$
\delta_j^{\beta}\delta_i^{\alpha}=\delta_i^{\alpha}\delta_{j-1}^{\beta} \ \ \text{for} \  i< j,
$$
$$
\sigma_j\sigma_i=\sigma_i\sigma_{j+1} \ \ \text{for} \  i\leq j,
$$
$$
\sigma_j\delta_i^{\alpha}=
\begin{cases}  \delta_i^{\alpha}\sigma_{j-1} &    \text{for} \  i< j\\
\delta_{i-1}^{\alpha}\sigma_{j} &    \text{for} \  i>j\\
\operatorname{Id}& \text{for} \  i=j.
\end{cases}
$$
\end{proposition}

\begin{definition}\label{d2.3}  \rm A \emph{cubical 
set} is a functor $K\colon \Box^{op}\to \operatorname{\bf Set}$
and a morphism of cubical sets is a natural transformation of functors.  We denote  by  $\bold{Cub}$ the category of cubical sets. 
\end{definition}

It follows from Definition \ref{d2.3} that a cubical set $K$  is defined by the family of sets 
$\{K_n= K(\Bbb I_n)\, | \, n\geq 0\}$, the face maps 
\begin{equation}
\partial^{\alpha}_i=K\left(\delta^{\alpha}_i\right)\colon K_n\to K_{n-1} \ (\alpha\in \Bbb I,   n\geq 1,  i=1, \dots ,n)
\end{equation}\label{2.1}
and the degeneracy maps 
\begin{equation}\label{2.2}
\varepsilon_i=K(\sigma_i)\colon K_{n-1}\to K_n \ (n\geq 1, \ i=1,\dots ,n)
\end{equation}
which satisfy the following relations 
\begin{equation}\label{2.3}
\begin{matrix}
\partial_i^{\alpha}\partial_j^{\beta}=\partial_{j-1}^{\beta}\partial_{i}^{\alpha} \ \ \text{for} \  i< j,\\
\varepsilon_i\varepsilon_j=\varepsilon_{j+1}\varepsilon_i \ \ \text{for} \  i\leq j,\\
\partial_i^{\alpha}\varepsilon_j=
\begin{cases} \varepsilon_{j-1} \partial_i^{\alpha} &    \text{for} \  i< j\\
\varepsilon_{j}\partial_{i-1}^{\alpha} &    \text{for} \  i>j\\
\operatorname{Id}& \text{for} \  i=j.
\end{cases}\\
\end{matrix}
\end{equation}
 We note that  morphisms of cubical sets are maps 
of sets  $K_n\to K_n^{\prime}$ for all $n$ which commute with face and degeneracy maps.

Let $T=[0,1]$ be the unit interval and $T^n$ be the standard $n$-cube for $n\geq 1$ and $T^0=\{0\}$. Then the maps 
$
\delta_i^{\alpha}\colon  T^{n-1}\to T^{n}  
$
and 
$\sigma_i\colon T^{n}\to T^{n-1}$
are defined similarly to definition \ref{d2.1} and they satisfy the equalities of   
Proposition \ref{p2.2}.

\begin{definition}\label{d2.4} \rm  A \emph{topological realization}  $|K|_{\bold{Top}}$ of a cubical set 
$K\colon \Box^{op}\to \bold{Set}$ is the quotient topological  space  
$$
 |K|_{\bold{Top}}= \{\coprod_n K_n\times T^n\}/\sim 
$$
where the set $K_n$ is equipped with the  discrete topology,  $T^n$ is the $n$-cube   and the  equivalence relation is generated  by 
$$
(\partial^{\alpha}_i(x), v))\sim (x, \delta_i^{\alpha}(v))   \ \text{for}  \ x\in K_n, 
 v\in T^{n-1}
$$
and
$$
 (\varepsilon_i(y), w)\sim (y, \sigma_i(w))  \ \text{for} \  y\in K_{n-1}, w\in T^n.
$$
\end{definition}

\begin{proposition}\label{p2.5}  {\rm \cite[Pr. 10.1.12]{Nonab}} The topological space 
$|K|_{\bold{Top}}$ has a natural structure of a $CW$-complex having one $n$-cell for each non-degenerate $n$-cube.  
The topological  realization $|\ |_{\bold{Top}}$  is  the functor  from the category $\bold{Cub}$ of cubical sets to the category $\bold{CW}$ of 
$CW$-complexes and cell-maps.  
\end{proposition}

Now we recall the standard notions  of the graph theory which we need in the paper (see, 
 \cite{MiHomotopy}). 

\begin{definition}
\label{d2.6}\rm \emph{A directed graph (digraph)}  $G=\left(
V_{G},E_{G}\right) $ consists of   a set $V_{G}$ of \emph{vertices} and  a
subset $E_{G}\subset \{V_{G}\times V_{G}\setminus \operatorname{diagonal}\}$ of ordered
pairs $(v,w)$ of vertices  which are called \emph{arrows}. The ordered pair of vertices  $(v,w)$ is  denoted  by
 $v\to  w$. The vertex $v=\mbox{orig }(v\rightarrow w)$ is called 
the \emph{origin of the arrow} and the vertex $w=\mbox{end}(v\rightarrow w)$
is called the \emph{end of the arrow}.
\end{definition}

 For $v,w\in V_{G}$, we write $v\,\overrightarrow{=}w$ if either $
v=w $ or $v\rightarrow w$. A digraph  $H$ is a  \emph{subgraph} of a digraph $G$  and we write  $H\subset G$ if  $V_H\subset V_G$ and $E_H\subset E_G$.  A \emph{directed path}
 in a digraph $G$ is a sequence of vertices  $a_i\in V_G\, (0\leq i\leq n)$ such that  $(a_{i}\to a_{i+1})\in E_G$ for $0\leq i\leq n-1$. The number $n$  of arrows
$a_{i}\to a_{i+1}$ fitting into the path is called the \emph{length} of the path.  The vertex $a_0$ is \emph{the origin of the path} and the vertex $a_{n}$ 
is \emph{the end of the path}.

\begin{definition}
\label{d2.7} \rm A  \emph{digraph map}  
$f\colon G\to H$ is  given by a map 
$f\colon
V_{G}\rightarrow V_{H}$  between  the sets  of vertices such that $v\overrightarrow{=} w$  in $G$  implies $f\left(
v\right) $\thinspace $\overrightarrow{=}f\left( w\right) $ in $H$. We call the map $f$ \emph{non-degenerate on an arrow}   
$(v\to w)\in E_G$ if $(f(v)\to f(w))\in E_H$.
\end{definition}

All digraphs with digraph maps form the  \emph{category of digraphs},
which  is denoted by $\bold{D}$.

\begin{definition}
\label{d2.8}\rm For two digraphs $G$ and $H$  define   the   \emph{Box  product}
$\Pi =G\Box H$ as the digraph with the set of
vertices $V_{\Pi }=V_{G}\times V_{H}$ and the set of arrows $E_{\Pi }$ given
by the rule
\begin{equation*}
(x,y)\rightarrow (x^{\prime },y^{\prime })\ \ \text{if }x=x^{\prime }\text{
and }y\rightarrow y^{\prime }\text{, or }x\rightarrow x^{\prime }\text{ and }%
y=y^{\prime },
\end{equation*}
where $x,x^{\prime }\in V_{G}$ and $y,y^{\prime }\in V_{H}$. 
\end{definition}

\begin{definition}\label{d2.9} \rm  Two digraph maps $f,g\colon
G\rightarrow H$ are\emph{\ }called \emph{homotopic} if there exists a line
digraph $I_{n}$ with $n\geq 1$ and a digraph map
\begin{equation*}
F\colon G\Box I_{n}\rightarrow H
\end{equation*}%
such that 
\begin{equation*}
F|_{G\Box \{0\}}=f\ \ \text{and}\ \ F|_{G\Box \{n\}}=g 
\end{equation*}
where we identify $G\Box \{0\}$ and $G\Box \{n\}$ with $G$ in a natural way. In this case we shall write $f\simeq g$. The map $F$ is called a \emph{
homotopy} between $f$ and $g$. In the case $n=1$ we refer to the map $F$ as an \emph{one-step homotopy}.
\end{definition}

\begin{definition}\label{d2.10}\rm Digraphs $G$ and $H$ are called \emph{homotopy equivalent
} if there exist  maps 
\begin{equation*}
f:G\rightarrow H,\ \ \ g:H\rightarrow G  \label{fGH}
\end{equation*}
such that 
\begin{equation*}
f\circ g\simeq \operatorname{id}_{H},\ \ \ \ \ g\circ f\simeq \operatorname{id}_{G}.
\label{fg}
\end{equation*}%
In this case we shall write $H\simeq G$ and  the maps $f$ and $g$ 
are called \emph{\ homotopy inverses }of each other.
\end{definition}

All digraphs with homotopic classes of digraph maps form the  \emph{homotopy category of digraphs} which  is denoted by $\bold{HoD}$~\cite{MiHomotopy}. 

\section{Category of quivers and $\Box$-categories of digraphs}\setcounter{equation}{0}
\label{S3}

We define a category of quivers  $\mathbf Q$ in such a way that the category of digraphs  $\mathbf D$ will be a subcategory of $\mathbf Q$ (see, for example, \cite{Enoch, MurVer}).

\begin{definition} \rm 
\label{d3.1}   A \emph{quiver}   $Q=(V,E,s,t)$ consists of a set of vertices 
$V$, a set of edges (arrows) $E$,  and two maps   $s,t\colon E\rightarrow V$.  The vertex    $s(a)\in V$
is called the start vertex of $a$,  and the vertex  $t(a)$  is called the end vertex of $a$. 

The map 
 $f:Q\rightarrow Q^{\prime }$ of quivers is given by a pair of maps 
  $f_{V}\colon V\rightarrow V^{\prime }$ and  
 $f_{E}\colon E\rightarrow E^{\prime }\cup V^{\prime}$ such that for every  $a\in E$ only one of the following  two conditions: 

i)  $f_E(a) \in E^{\prime}$ and $f_{V}(s(a))=s^{\prime }(f_{E}(a))$,      $f_{V}(t(a))=t^{\prime}(f_{E}(a))$,  

ii)  $f_E(a)=v^{\prime} \in V^{\prime}$ and $f_{V}(s(a))=f_{V}(t(a))=v^{\prime}$

\noindent is  satisfied. We denote by  $\mathbf  Q$ the category of quivers and their maps. 
\end{definition}

 Every digraph $G=(V, E)$ defines a  quiver $Q=(V,E,s,t)$,  where $s(v\to w)=v, t(v\to w)=w$ for $(v\to w)\in E$.  
Every map $f\colon G\to G^{\prime}$ of digraphs defines the map $f=(f_V, f_E)$  of the corresponding quivers as follows. 
The map  $f_V$  coincides with  the map $f\colon V\to V^{\prime}$ on the set  of quiver   vertices   and 
$$
f_{E}(v\to w)=\begin{cases} f(v)\to f(w) & \text{for}\  f(v)\ne f(w),\\
                                            f(v)&  \text{for}\  f(v)= f(w).\\
\end{cases}
$$
 Hence,  we have  the inclusion of categories  $\mathbf D\subset \mathbf Q$.

Let us fix $n\geq 0$ and  denote by   $I_{n}$ a digraph with the  set of
vertices $V=\{0,1,\dots ,n\}$ and such that   for $i=0,1,\dots n-1$,  there is exactly one arrow $i\to i+1$ or $i+1\to i$ 
and there are no other arrows.  Such digraph is  called a  \emph{line} digraph and a \emph{direct line} digraph if additionally 
all arrows are of  the form of $i\to i+1$.  We denote the line digraph $0\to 1$ by $I$. Define a category $\bold{I}$ of 
line digraphs and digraph maps. It has a subcategory $\bold{DI}$ of directed line digraphs.

Let us fix a line digraph $I_k (k\geq 1)$. For $n \geq 0$, we define  an \emph{$n$-cube digraph}  $I^n_k$  as
follows. For $n = 0$ we put $I^0_k = \{0\} $ --- one-vertex digraph and  for $n \geq 1$, 
$I^n_k$ is given by  
$$
I^n_k = \underbrace{I_k\Box I\Box I_k\Box \dots \Box I_k}_{n-times}.
$$
Note, that for $n\geq 0$, every digraph morphism $\tau\colon I_k\to I_m$ induces a digraph morphism 
$\tau^n\colon  I_k^n\to I_m^n$.

For    $n\geq 2, \ 1\leq i\leq n$, and $\alpha\in \{0,k\}\subset  V_{I_k}$, we define the
following inclusion  $\delta_{i}^{\alpha}\colon I^{n-1}_k\rightarrow I^{n}_k$ of the face of digraphs given on the set of vertices by 
\begin{equation}\label{3.1}
\begin{matrix}
\delta_{i}^{\alpha }(c_{1},\dots ,c_n)=(c_{1},\dots ,c_{i-1},\alpha
,c_{i},\dots , c_n) \ \text{for} \  \ c_j\in V_{I_k}.
\end{matrix}
\end{equation}
We define also $\delta_1^{\alpha}\colon I_k^0=\{0\}\to I_k^1=I_k$ by  $\delta_{1}^{\alpha }(0)=\left( \alpha \right)$.

For  $n\geq 2$  and $1\leq i\leq n$,  let us consider the  natural projection 
$\sigma_i\colon I^{n}_k\to I^{n-1}_k$ on the $i$-face $I^{n-1}_k$ given on the set of vertices by 
\begin{equation}\label{3.2}
\sigma_i(c_1,\dots, c_{n})=(c_1,\dots, c_{i-1},c_{i+1}, \dots,c_{n})  \ \text{for} \  \ c_j\in V_{I_k}, 
\end{equation}
 and let 
$\sigma_1\colon I^1_k\to \{0\}$  be given on the set of  vertices  by $\sigma_1(c_1)=0$ for $c_1\in V_{I_k}$.  
For $k\geq 1$ denote by  $\Bbb D_k$ the category  consisting of $n$-cubes $I^n_k$ and  maps obtained by compositions of the maps $\delta_i^{\alpha}, \sigma_i$ defined above.

\begin{proposition}\label{p3.2}  For fixed $k\geq 1$ and every line digraph $I_k$, the category $\Bbb D_k$, 
with objects $I^n_k$ and  morphisms generated  by $\sigma_i$, $\delta_i^{\alpha}$,   is isomorphic to the $\Box$-category.
\end{proposition}

  Since we have the inclusion of categories $\mathbf D\subset \mathbf Q$,  we can consider every digraph as a quiver 
if this not leads to confusion.  In order to introduce the notion of homotopy in  $\mathbf Q$  we  reformulate Definition   \ref{d2.9} for this category. 

\begin{definition} \label{d3.3}\rm For a line digraph $I_k\, (k\geq 0)$ and a quiver $Q=(V,E,s,t)$ we define   the   \emph{Box  product}
$\Pi =I_k\Box \,Q=(V_{\Pi}, E_{\Pi}, s_{\Pi},t_{\Pi})$ as the quiver  with the set of
vertices $V_{\Pi }=V_{I_k}\times V_{H}$,  the set of arrows 
$$
E_{\Pi }=\{(v, a)| v\in V_{I_k}, a\in E\}\cup
 \{[(v,w)\to (v^{\prime},w)]\,|\, ( v\to v^{\prime})\in E_{I_k}, w\in V\}, 
$$
 and the maps $s_{\Pi},t_{\Pi}$ given
by the rule
\begin{equation*}
\begin{matrix}
s_{\Pi}(v, a)=(v, s(a)), \  t_{\Pi}(v, a)=(v, t(a)),\\
s_{\Pi}[(v,w)\to (v^{\prime},w)]= (v,w), \  t_{\Pi}[(v,w)\to (v^{\prime},w)]= (v^{\prime},w). 
\end{matrix}
\end{equation*}
where  $v, v^{\prime}\in V_{I_k}, a\in E, ( v\to v^{\prime})\in E_{I_k}, w\in V$.
\end{definition}

 \begin{definition} \label{d3.4}\rm Two quiver  maps $f,g\colon
Q\rightarrow Q^{\prime}$ are\emph{\ }called \emph{homotopic} if there exists a line
digraph $I_{k}$  and a quiver map
\begin{equation*}
F\colon I_{k}\Box Q\rightarrow Q^{\prime}
\end{equation*}%
such that 
\begin{equation*}
F|_{\{0\}\Box Q}=f\ \ \text{and}\ \ F|_{\{k\}\Box Q}=g 
\end{equation*}
where we identify $\{0\}\Box Q$ and $\{k\}\Box Q$ with $Q$ in a natural way. In this case we shall write $f\simeq g$. The map $F$ is called a \emph{
homotopy} between $f$ and $g$. In the case $n=1$ we call the map $F$  a \emph{one-step homotopy}.
\end{definition}

\begin{definition}\label{d3.5} \rm  A quiver  $Q=(V,E,s,t)$ is \emph{simple} if the following two conditions are satisfied: 

 i) $s(a)\ne t(a)$ for every arrow  $a\in E$,

ii)   ordered pairs of vertices  $(s(a), t(a))$ and   $(s(b), t(b) )$ are different for every two different  arrows  $a, b\in E$. 
\end{definition}

\begin{remark}\label{r3.6} \rm  It is clear that every simple quiver $Q=(V,E,s,t)$ defines a digraph $G=(V_G,E_G)$ where $V_G=V$, $E_G=\{s(a)\to t(a)\,|\, a\in E\}$. 
\end{remark}

\section{Singular cubical set of a quiver and the  quiver realization of a cubical set}
\setcounter{equation}{0}
\label{S4}

Now,  for every quiver $Q$ and every line  digraph $I_k$,  we define a  cubical set which functorially depends of  $Q$ and  $I_k$. 
More precisely, every line digraph $I_k$ defines a cubical set $S_{I_k}^{\Box}(Q)$ consisting in dimension 
$n\geq 0$ of singular digraph cubes which are  given by quiver maps $\varphi\colon I_k^n\to Q$. Any morphism of quivers $f\colon Q\to Q^{\prime}$ induces 
the morphism   $f_{\Box}=f\varphi\colon S_{I_k}^{\Box}(Q)\to S_{I_k}^{\Box}(Q^{\prime})$, and any morphism $\tau\colon I_k\to I_m$ induces the morphism 
$\tau^{\Box}=\varphi\tau \colon S_{I_m}^{\Box}(Q)\to S_{I_k}^{\Box}(Q)$. 
 Hence  singular cubical sets of quivers give a collection of covariant functors from the category $\bf Q$ to the category $\bold{Cub}$ and a collection  of contravariant functors
 from the category $\bold I$ to the category $\bold{Cub}$. We describe relations between obtained functors and give  applications of these results to the category of digraphs  $\mathbf D\subset \mathbf Q$. 

 \begin{remark}\label{r4.1} \rm For a topological space $X$,  the singular cubical set   $S^{\Box}(X)$  is given by singular topological cubes in $X$ (see  \cite{Antolini}, \cite{F},  \cite{Kan},  \cite{Massey}, \cite{Serre}).  The definition of singular cubical set  $S_{I_k}^{\Box}(Q)$ of a quiver $Q$ given below  is similar to the definition  $S^{\Box}(X)$ of a topological space $X$. Recall that for a  topological space $X$ the topological realization $\left| S^{\Box}(X)\right|_{\bold{Top}}$ and $X$ are week homotopy  equivalent.
\end{remark}

Fix a line digraph $I_k \  (k\geq 1)$. For $n\geq 0$, \emph{a singular $n$-cube in  a quiver $Q$} is a quiver  map $\phi\colon I^n_k \to Q$. For any quiver $Q$ denote by $S_{I_k}^{\Box}(Q)=\left\{\left[S_{I_k}^{\Box}(Q)\right]_n\left|\, \right. n\geq 0\right\}$ the set of all singular cubes  $\phi\colon I^n_k \to Q\, (n\geq 0)$. 

\begin{proposition}\label{p4.2} Let $Q$ be a quiver.  For $n\geq 1$  define  the face morphism 
$\partial_i^{\alpha} \colon \left[S_{I_k}^{\Box}(Q)\right]_n \to  \left[S_{I_k}^{\Box}(Q)\right]_{n-1}$  by  $\partial_i^{\alpha}(\phi)=\phi \delta^{\alpha}_i$ and define the degeneracy  morphism  $\varepsilon_i  \colon \left[S_{I_k}^{\Box}(Q)\right]_{n-1} \to \left[S_{I_k}^{\Box}(Q)\right]_{n}$ by
 $\varepsilon(\phi)= \phi\sigma_i$.  Thus  $S_{I_k}^{\Box}(Q)$ is   a cubical set. 
In such a way,  for any line digraph $I_k$, we obtain a singular cubical  functor 
$$
S_{I_k}^{\Box}\colon \bold{Q}\to \bold{Cub}. 
$$
The restriction of this functor to the subcategory of digraphs $\mathbf D\subset \mathbf Q$ provides 
a functor  $\bold{D}\to \bold{Cub}$ which we continue to denote $S_{I_k}^{\Box}$. 
\end{proposition}

\begin{remark}\label{r4.3} \rm In the case of directed line digraphs singular cubes were introduced in \cite{Mi4} and \cite{Mi} for the construction of singular cubical homology groups and cell singular cubical homology groups of digraphs. 
\end{remark}

\begin{proposition}\label{p4.4} For any quiver $Q$, we have a contravariant functor \ 
$
\mathcal I\colon \bold I\to \bold{Cub}
$
defined on the objects by $\mathcal I(I_k)= S_{I_k}^{\Box}(Q)$,   and for any digraph morphism 
$\tau\colon I_k\to I_m$ we define 
$
\mathcal I(\tau)\colon S_{I_m}^{\Box}(Q)\to S_{I_k}^{\Box}(Q)
$
for $\phi\colon  I^n_m\to Q$ putting $\mathcal I(\tau)[\phi]= \phi\circ \tau^n$.  
\end{proposition}

\begin{proof}
For $n\geq 0$, every digraph morphism $\tau\colon I_k\to I_m$ induces a digraph morphism 
$\tau^n\colon  I_k^n\to I_m^n$. 
\end{proof}

Note that for any set $K$ and a quiver  $Q$ we can define a quiver $K\times Q$ that is given by a disjoint union of  copies of the quiver   $Q$  enumerated by elements of $K$. 

Let  $f\colon Q\to Q^{\prime}$ be a map of quivers. 
We define a relation $R$ on $V\coprod V^{\prime}$ by the formula
$$  (v, v^{\prime})\in R \ \text{if}  \ f_V(v)=v^{\prime}.
$$
Let $\sim_f $  be the  equivalence relation  generated by the relation $R$, that is  $x\sim_f y$  if and only if there exist $n\geq 0$ and
$x_k$, $0 \leq k \leq n$, 
  such that $x=x_0$, $y=x_n$, and for every $1\leq k\leq n$, either $(x_{k-1},x_k)$ or $(x_k, x_{k-1})$ belongs to $ R$. In the same way we define on $\left[E\setminus  \left\{a\in E| f_E(a)=v^{\prime}\in V^{\prime}\right\}\right]\coprod E^{\prime}$
a relation $S$ by the formula
$$  (a, a^{\prime})\in S \  \ \text{if}  \ f_E(a)=a^{\prime}.
$$
Let $\sim_f $  be the  equivalence relation on $\left[E\setminus  \left\{a\in E| f_E(a)=v^{\prime}\in V^{\prime}\right\}\right]\coprod E^{\prime}$ generated by the relation $S$.

  Using the map $f$,  we define a quiver  $\sum_f=\sum_f(Q,Q^{\prime})=(V_{f}, E_{f}, s_{f}, t_{f})$ 
by setting   
$$
V_{f}=V\coprod V^{\prime}/\sim_f,  \ \   
E_{f}=\left[E\setminus  \left\{a\in E| f_E(a)=v^{\prime}\in V^{\prime}\right\}\right]\coprod E^{\prime}/\sim_f, 
$$ 
and the maps  $s_{f}, t_{f}$ are induced by the corresponding maps  in  $Q$  and  in $Q^{\prime}$. We call  the quiver $\sum_f$ by  the \emph{quotient quiver defined by the equivalence relation  $\sim_f$}.

 We note that the product $*\times G$ of a digraph $G$ with the one element set $*$ is naturally identified with $G$.  Let $K$ be a cubical set and  $x\in K_n$,   $y=\partial_i^{\alpha}(x)\in K_{n-1}$. We have the map 
$\Gamma(x, i, \alpha)\colon y \times I^{n-1}\to x\times I^n$  which is defined by  $\delta_i^{\alpha}\colon I^{n-1}\to I^{n}$. Thus we obtain a quotient quiver 
$
\sum_{\Gamma(x, i, \alpha)}(y\times I^{n-1}, x\times I^n)
$
which is defined by the equivalence relations  $\sim_{\Gamma(x, i, \alpha)}$. 
  For $y\in K_{n-1}$ and   $x=\varepsilon_i(y)\in K_{n}$, we have  the map   $\Theta(y, i)\colon x \times I^{n}\to y\times I^{n-1}$ which is given by the map $\sigma_i\colon I^{n}\to I^{n-1}$.  Thus we obtain a quotient quiver 
$
\sum_{\Theta(y, i)}(y\times I^{n-1},x\times  I^n)
$
which is defined by the equivalence relations  $\sim_{\Theta(y, i)}$. 

\begin{definition}\label{d4.5} \rm \cite{MurVer} A \emph{simple quiver  realization}  $|K|_{\bold Q}$ of a cubical set 
$K$ is the  quiver
$$
 |K|_{\bold Q}= \{\coprod_n K_n\times I^n\}/\sim 
$$
where the equivalence relation is generated by equivalence relations $\sim_{\Gamma( x,i,\alpha)}$ and  $\sim_{\Theta(y,i)}$ defined above. 
\end{definition}

For a line digraph $I_k\, (k\geq 1)$ we can define the equivalences relations $\sim_{\Gamma( x,i,\alpha)}$ and $\sim_{\Theta(y,i)}$ on the sub-quivers  of $\coprod_n K_n\times I^n_k$ 
similarly to  the case $I=I_1$ above. 

\begin{definition}\label{d4.6} \rm  Let $I_k$ be a line digraph with $k\geq 1$. An \emph{$I_k$-quiver realization}  $|K|^{I_k}_{\bold Q}$ of a cubical set 
$K\colon \Box^{op}\to \bold{Set}$ is the  quiver 
$$
 |K|_{\bold Q}^{I_k}= \{\coprod_n K_n\times I^n_k\}/\sim 
$$
where the  equivalence relation is generated by equivalence relations $\sim_{\Gamma( x,i,\alpha)}$ and  $\sim_{\Theta(y,i)}$  for the sub-quivers of $\coprod_n K_n\times I^n_k$. 
\end{definition}

Note that it follows directly from these definitions that  $|K|_{\bold Q}^{I_1}\cong  |K|_{\bold Q}$ for any line digraph $I_1$ which can be equal to $0\to 1$ or $0\leftarrow 1$.

\begin{proposition}\label{p4.7} The quiver realizations $|\ |_{\bold Q}$ and $|\ |^{I_k}_{\bold Q}$ are functors from the category $\bold{Cub}$ of cubical sets to the category 
$\bold Q$ of quivers.
\end{proposition}

For a quiver $Q$, let $S^{\Box}_{\bold Q}(Q)=S^{\Box}_{I_1}(Q)$ where $I_1=(0\to 1)$. 
For two quivers $Q, Q^{\prime}$, we denote by $\mathbf Q(Q, Q^{\prime})$ the set of maps from 
$Q$ to $Q^{\prime}$ in the category $\mathbf Q$. For two cubical sets $K, K^{\prime}$, we denote by 
$\mathbf{Cub}(K, K^{\prime})$ the set of maps from  $K$ to   $K^{\prime}$ in the category 
$\mathbf{Cub}$. 

\begin{theorem}\label{t4.8} The quiver realization functors  $|\ |_{\bold Q}$ and $|\ |^{I_k}_{\bold Q}$ are left adjoint to the singular cubical  functors $S^{\Box}_{\bold Q}$ and
$S^{\Box}_{I_k}$, respectively. This means that for every cubical set $K$ and every  quiver $Q$ there are natural bijections
\begin{equation}\label{4.1}
E\colon \bold Q(|K|_{\bold Q}, Q)\to \bold{Cub}\left(K, S^{\Box}_{\bold Q}(Q)\right)
\end{equation}
and 
\begin{equation}\label{4.2}
E_{I_k}\colon \bold Q(|K|_{\bold Q}^{I_k}, Q)\to \bold{Cub}\left(K, S^{\Box}_{I_k}(Q)\right).
\end{equation}
\end{theorem}

\begin{proof} It is sufficient to prove (\ref{4.1}) since the proof of (\ref{4.2}) is similar. 
Let $f\colon |K|_{\bf Q}\to Q$ be a quiver map. Define the cubical map $E(f)\colon K\to S^{\Box}_{\bold Q}(Q)$  as follows. For every element  $x\in K_n\ (n\geq 0)$ we define the singular 
 cube $E(f)(x)\colon I^n\to Q$ for the quiver $Q$ by setting 
$$
[E(f)(x)](v)=f(x, v) \ \text{where} \  x\in K_n, v\in V_{I^n}  
$$
and 
$$
[E(f)(x)](a)=f(x, a) \ \text{where} \  x\in K_n, a\in E_{I^n}  
$$
The map $E$ is well defined and has an inverse map $E^{-1}$ which is defined as follows. 
Let $g\colon K\to S^{\Box}_{\mathbf Q}(Q)$ be a cubical map. For every vertex 
$(x, v)\in |K|_{\bold Q}$ where $x\in K_n$, $v\in V_{I^n}$ we set 
$
[E^{-1}(g)](x,v)=[g(x)](v)
$ and for every arrow 
$(x, a)\in |K|_{\bold Q}$ where $x\in K_n$, $a\in E_{I^n}$ we set 
$
[E^{-1}(g)](x,a)=[g(x)](a)$.
\end{proof}

Let us consider  the line digraph $I=I_1=(0\to 1)$ and a line digraph $I_k$ with $k\geq 1$. 
 Let $m,m+1\in V_{I_k}$   where $0\leq m<m+1\leq k$ be a pair of consecutive vertices.  Depending of the arrow between 
these vertices in $I_k$,  we have  the inclusion map  $i\colon I\to I_k$ given on the set of vertices by $i(0)=m, i(1)=m+1$ or 
$i(0)=m+1, i(1)=m-1$. Depending of this  inclusion,  define the projection $p\colon I_k\to I$ on the set of vertices $v\in V_{I_k}$
by  setting
$
 p(v)=\begin{cases} 0 & \text{for}\  v\leq m\\
                                                                        1 & \text{for} \ v\geq m+1\\
\end{cases}
$
for the first  case and 
$
 p(v)=\begin{cases} 1 & \text{for}\  v\leq m\\
                                                                      0 & \text{for} \ v\geq m+1\\
\end{cases}
$
for the second case, respectively.   For any $n\geq 1$ these maps 
induce the digraph maps 
\begin{equation*}\label{}
p^n\colon I_k^n\to I^n \ \ \text{and} \ \ i^n\colon I^n\to I_k^n
\end{equation*}
and we define the morphisms  $p^0$ and $i^0$ as   identity map $0\to 0$. It follows immediately from the definition that $p^ni^n=\operatorname{Id}$.

\begin{proposition}\label{p4.9} For any  quiver  $Q$,  the digraph maps $p$ and $i$ induce the following maps of singular  cubical sets 
$$
 p^{\Box}\colon S^{\Box}_{\bold{Q}}(Q)\to S^{\Box}_{I_k}(Q) \ \
\text{and} \ \ \
i^{\Box}\colon S^{\Box}_{I_k}(G)\to S^{\Box}_{\bold Q}(G)
$$ 
such that  $i^{\Box}p^{\Box}=\operatorname{Id}\colon S^{\Box}_{\bold{Q}}(Q)\to S^{\Box}_{\bold{Q}}(Q)$.
\end{proposition}
\begin{proof} It follows from  the definitions of the  singular cubical sets  and from the relation  $p^ni^n=\operatorname{Id}\colon I^n\to I^n$ for $n\geq 0, k\geq 1$.
\end{proof}

\begin{definition}\label{d4.10} \rm  A cubical set $K$ is \emph{simple} if the following two conditions are satisfied: 

 i) $\partial_1^1(k)\ne \partial_1^0(k)$ for every cube $k\in K_1$,

2)   ordered pairs  $(\partial_1^1(k), \partial_1^0(k))$ and   $(\partial_1^1(k^{\prime}), \partial_1^0(k^{\prime}))$ are different for every two different 
cubes $k, k^{\prime}\in K_1$. 
\end{definition}

\begin{definition}\label{d4.11} \rm \cite{Jardine} For a cubical set $K$ and $q\geq 0$  we define the \emph{$q$-skeleton}
 $\operatorname{sk}_q(K)$ 
of $K$ as the cubical set which is generated by the $n$-cubes of $K_n$ for $0\leq n\leq q$. 
\end{definition} 

Note that we have the filtration 
\begin{equation}\label{4.3}
\operatorname{sk}_0(K)\subset  \dots \subset \operatorname{sk}_q(K)\subset \dots \subset K
\end{equation}
  and the cubical set 
$\operatorname{sk}_q(K)$ has no non-degenerate cubes in dimensions greater then $q$. Thus for any quiver $Q$ we have the following filtration 
\begin{equation}\label{4.4}
\operatorname{sk}_0\left( S^{\Box}_{I_k}(Q)\right)\subset  \dots \subset \operatorname{sk}_q\left(S^{\Box}_{I_k}(Q)\right)\subset \dots \subset S^{\Box}_{I_k}(Q)
\end{equation}
of the cubical set $S^{\Box}_{I_k}(Q)$ where $\operatorname{sk}_q\left(S^{\Box}_{I_k}(Q)\right)$ is generated by non-degenerate singular cubes with dimension less or equal to $q$.  

\begin{theorem} \label{t4.12} For any cubical set  $K$ the quiver  realization 
$\left|K\right|_{\mathbf Q}$ 
coincides with the quiver realization   $\left|\operatorname{sk}_1 (K)\right|_{\mathbf Q}$
of the one-dimensional skeleton  of the cubical set $K$.
\end{theorem}
\begin{proof} It follows  from  Definition \ref{d4.5} that the quiver realization 
$|K|_{\bold Q}$ is defined by the set of non-degnerate cubes of $K$.   The cubical set 
$\operatorname{sk}_q (K)$ contains only degenerate cubes in dimensions greater then $q$.  
Let us consider  the skeleton filtration of $K$ given by  (\ref{4.3}) and let us apply induction on  $q$ the number of this filtration.  If the cubical set $K$ has non-degenerate cubes only in dimensions  zero  and one 
then $K=\operatorname{sk}_1(K)$ and it is nothing to prove.   
Let us look at  the  induction step. We have 
$|\operatorname{sk}_{q-1}(K)|_{\bold Q}\subset |\operatorname{sk}_{q}(K)|_{\bold Q}$. For every non-dgenerate cube $x=k_q\in \operatorname{sk}_{q}(K)$ and for every  
$k_{q-1}=y=\partial_i^{\alpha}(x)\in \operatorname{sk}_{q-1}(K)$  
we have the map $\Gamma(x,i, \alpha)\colon y\times I^{q-1}\to x\times I^q$ which 
is defined by the inclusion $\delta_i^{\alpha}\colon I^{q-1}\to I^q$ of the face of the cube $I^q$. Every $y=\partial_i^{\alpha}(x)\in\operatorname{sk}_{q-1}(K)$ defines an inclusion  $\Gamma(x,i, \alpha)$ and  the cube $I^q$ is the union of its $(q-1)$-faces.  
Hence, by the definition of the equivalence relation $\sim_{\Gamma(x,i, \alpha)}$  we obtain 
that the quiver realization of $\left|\operatorname{sk}_{q}(K)\right|_{\bold Q}$ coincides with the quiver realization of $\left|\operatorname{sk}_{q-1}(K)\right|_{\bold Q}$ 
which coincides with  $\left|\operatorname{sk}_1 (K)\right|_{\mathbf Q}$ by the  inductive assumption. 
\end{proof}

\begin{corollary} \label{c4.13} For any quiver $Q$ the quiver  realization 
$\left|S^{\Box}_{I_k}(Q) \right|_{\mathbf Q}$
of the  cubical set  $S^{\Box}_{I_k}(Q)$ 
coincides with the quiver realization   $\left|\operatorname{sk}_1\left( S^{\Box}_{I_k}(Q)\right)\right|_{\mathbf Q}$
of the one-dimensional skeleton  of  $S^{\Box}_{I_k}(Q)$.
\end{corollary}

\begin{corollary} \label{c4.14} For any quiver $Q$ the quiver  realization 
$\left|S^{\Box}_{\bf Q}(Q) \right|_{\mathbf Q}$ of the cubical set $S^{\Box}_{\bf Q}(Q)$ coincides with the quiver $Q$. 
\end{corollary}
\begin{proof} This statement is proved in \cite{MurVer} using the same line of arguments as 
in the proof of Theorem  \ref{t4.12}.
\end{proof}

\begin{proposition}\label{p4.15}  i)  The simple quiver realization $|K|_{\bold Q}$ of a simple cubical set $K$ is a digraph. 

ii)  For any cubical set $K$ and any line digraph $I_k$ with $k\geq 2$ the $I_k$ quiver realization $|K|_{\bold Q}^{I_k}$ is a digraph. 
\end{proposition}
\begin{proof} i) It follows  from  Definition \ref{d4.5} that the simple quiver realization 
$|K|_{\bold Q}$ is defined by the set of non-degnerate cubes of $K$.   
Now we use induction in $q$ for the skeleton filtration of $K$ given by (\ref{4.3}).  For $q=0, 1$ 
the simple quiver  realization  of $\operatorname{sk}_q(K)$ gives  a simple quiver
 by  Definitions \ref{d4.10}  and \ref{d3.5}   which is 
a digraph by Remark \ref{r3.6}.  To obtain the  step of induction it is sufficiently to note 
that the quiver realization of an $q$-dimensional cube $k_q\in K_q$ is obtained by attaching $q$-dimensional digraph cube
 $I_k^q$  to the quiver realization  of its faces  which lay in the quiver realization of  $\operatorname{sk}_{q-1}(K)$ by 
 the equivalence relation $\sim_{(\delta, x,i,\alpha)}$ from  Definition \ref{d4.6}. 

ii)  It follows immediately from consideration  of $I_k$ realization of $1$-dimensional skeleton 
$\operatorname{sk}_1(K)$. 
\end{proof}

\begin{example}\label{e4.15} \rm  Now we show    that in opposite to Corollary  \ref{c4.14}, the quiver realization 
$\left|S^{\Box}_{I_k}\right|_{\bold Q}$ does not coincide with the quiver $Q$  in general case for  $k\geq 1$. 
Let $G=(0\to 1)$ be the digraph $I$  and $I_2=(0\to 1\to 2)$ be the line digraph. We describe the skeletons 
$\operatorname{sk}_q(K)\, (q=0,1)$  of the cubical set $S^{\Box}_{I_2}(G)$.  To do this it is sufficiently to
 describe all  non-degenerate singular cubes $I_2^q\to G$ for $q=0,1$. For $q=0$ we have  two such 
cubes $\psi^0_i\colon  \{0\}\to G\, (i=0,1)$ where $\psi^0_i(0)=i\in V_G$. For $q=1$ 
we have  also two non-degenerate  cubes $\phi^1_i\colon  I_2\to G\, (i=0,1)$  given on the set of vertices by 
 $$
\phi^1_0(j)= \begin{cases} j& \text{for}\ j=0,1, \\
                                           1& \text{for}\ j=2\\
\end{cases} 
\ \ \text{and }
\ \ 
\phi^1_1(j)= \begin{cases} 0& \text{for}\ j=0,1, \\
                                           1& \text{for}\ j=2.\\
\end{cases} 
$$ 
It is easy to see, that $\partial^1_1(\phi^1_0)=\psi^0_1$, $\partial^1_0(\phi^1_0)=\psi^0_0$,  $\partial^1_1(\phi^1_1)=\psi^0_1$, $\partial^1_0(\phi^1_1)=\psi^0_0$. Consider  the quiver realization  
$\left|\operatorname{sk}_1\left( S^{\Box}_{I_2}(G)\right)\right|_{\mathbf Q}$
of the one-dimensional skeleton  of the cubical set $S^{\Box}_{I_2}(G)$. It  is the quiver $Q=(V,E,s,t)$ with two vertices  
$(\psi^0_0, 0)$ and $(\psi^0_1, 0)$, two arrows $(\phi^1_0, I_1)$ and  $(\phi^1_1, I_1)$ where $s(\phi^1_0, I_1)= s(\phi^1_1, I_1) =(\psi^0_0, 0), t(\phi^1_0, I_1)= t(\phi^1_1, I_1) =(\psi^0_1, 0)$. The quivers $Q$ and $G$ are different. Hence the quivers $\left| S^{\Box}_{I_2}(G)\right|_{\mathbf Q}$ and $G$  are different.
\end{example}

\section{Homology of  cubical sets and  quivers}
\setcounter{equation}{0}
\label{S5}

In this Section we recall the definition and basic properties of homology groups of cubical sets \cite{Nonab, Grandis1}. Then we 
define the collection of singular cell homologies  of a quiver and describe relations between them. Using the quiver realization of cubical sets we introduce also path homology of cubical sets and describe their properties.

 Let $R$ be a commutative and unitary ring of coefficients and let $K$ be a cubical set. For $n\geq 0$, let $Q_n(K)$ be the free $R$-module generated by elements $k_n\in K_n$ and let $Q_{-1}(K)=0$. For $n\geq 1$, define the map $\partial\colon Q_n(K)\to Q_{n-1}(K)$ on basic 
elements $k_n$ setting 
\begin{equation}\label{5.1}
\partial(k_n)=\sum_{i=1}^n (-1)^i
\left(\partial_i^0(k_n)-\partial_i^1(k_n)\right), 
\end{equation}
and we have  $\partial=0\colon  Q_0(K)\to Q_{-1}(K)=0$. It is easy to check that $\partial^2=0$.  Let 
$B_{-1}(K)=B_{0}(K)=0$ 
and,  for $n\geq 1$, $B_n(K)$ be the submodule  of $Q_n(K)$ generated by 
degenerated elements $\varepsilon(k_{n-1})$ where $k_{n-1}\in K_{n-1}$. We have  
$\partial(B_n)\subset B_{n-1}$ for 
$n\geq 0$.  Hence we obtain  a chain complex $(Q_*(K), \partial)$ and a quotient 
complex $\Omega_*^c(K)=Q_*(K)/B_*(K) $ with the induced differential which 
we continue to denote $\partial$. 

Denote by $\bold{Ch}$ the category consisting of chain complexes of modules over  $R$ and chain homomorphisms. It follows immediately from the definition of the chain complex $\Omega_*^c(K)$  that we have a functor
$
\mathcal C\colon \bold{Cub}\longrightarrow \bold{Ch}.
$

\begin{definition}\label{d5.1} \rm The homology groups   $H_*(\Omega_*^c(K), R)$ are called \emph{normalized homologies } of the cubical set $K$ 
and are denoted $H_*(K,R)$. 
\end{definition}

In what follows we shall omit the ring $R$ in the notations of homology and  we shall write $H_*(K)$ instead of $H_*(K,R)$ to simplify notations.

There is  a relation between homology groups of a cubical set $K$ and its topological  realization $ |K|_{\bold{Top}}$  \cite[\S 1.8]{Grandis1}.

\begin{proposition}\label{p5.2} The  homology groups $H_*(K)$ of a cubical set
$K$ coincide with the homology  groups $H_*(|K|_{\bold{Top}})$ of the $CW$-complex  
$|K|_{\bold{Top}}$. 
\end{proposition}

Now,   for any line digraph $I_k\, (k\geq 1)$  we define  the functor of  cell homology groups on the category    $\bold Q$  of quivers. 
In  the case of directed line digraphs  the cell homology groups on the category $\mathbf D$ of digraphs where defined in \cite{Mi4, Mi}.  
We note that the category $\mathbf Q$ of quivers   differs from the category $\mathcal Q$ of quivers  that is considered  in \cite{Mi4}.  
The objects of these categories are the same and we have the inclusion
$\mathcal Q\subset \mathbf Q$ of categories since   in $\mathbf Q$ we admit morphisms which satisfy  the  condition  ii)  in Definition \ref{d3.1}. 

\begin{definition}\label{d5.3} \rm Let $Q$ be a quiver  and $I_k$ be a line digraph with $k\geq 1$. Define the \emph{cell homology groups}
$
H_*^{I_k}(Q)
$
setting 
$$
H_*^{I_k}(Q)=H_*\left(S^{\Box}_{I_k}(Q)\right). 
$$
\end{definition}

\begin{corollary}\label{c5.4} For any quiver $Q$ and any line digraph $I_k$ with $k\geq 1$, the  homology groups 
$
H_*^{I_k}(Q)
$ 
 coincide with the homology  groups 
$H_*\left(\left|S^{\Box}_{I_k}(Q)\right|_{\bold{Top}}\right)$ of the $CW$-complex 
 $\left|S^{\Box}_{I_k}(Q)\right|_{\bold{Top}}$.
\end{corollary}

\begin{theorem}\label{t5.5} Let $f\simeq g\colon Q\rightarrow Q^{\prime}$ be two homotopic quiver
maps. Then 
\begin{equation*}
f_{\ast }=g_{\ast }\colon H_{m}^{I_k}(Q)\rightarrow H_{m}^{I_k}(Q^{\prime})\ \ \text{for
any }\ \ m\geq 0.
\end{equation*}
\end{theorem}
\begin{proof}   For $k=1$ the proof  is standard \cite[Theorem 8.3.8]{Hilton}.   The general case is similar 
to the case of directed line digraph $I_k$  and the maps $f$, $g$ in the category $\mathbf D$ which was 
considered in  \cite{Mi}. Now we shortly recall this proof  in the notations of the present paper.    
It is sufficient to prove the Theorem in  the case of one-step homotopy with $I_1=I=(0\to 1)$:
\begin{equation*}
F\colon  I\Box Q\to Q^{\prime}, \ F_{\{0\}\Box Q}=f, \ \ F_{\{1\} \Box Q}=g.
\end{equation*}
Without restriction of generality we suppose that the digraph $I_k$ contains the arrow $(0\to 1)$.
Define the digraph map $\pi\colon I_k\to I=(0\to 1)$ setting on the set of vertices  
$ \pi(0)=0$ and  $\pi(v)=1$ for $v\ne 0$. Let $\Pi=\pi\colon I_k\to I$.  For $n\geq 1$, we 
 define $\Pi=\operatorname{Id} \Box \pi\colon I_k^{n+1}= I_k^{n}\Box I_k\to I_k^{n}\Box I$ where   
$\operatorname{Id}\colon   I_k^{n}\to  I_k^{n}$ is the identity map.
 Define the sequence of homomorphisms 
$$
s_n\colon {\Omega}_n^c\left(S^{\Box}_{I_k}(Q)\right)\to {\Omega}_{n+1}^c\left(S^{\Box}_{I_k}(Q^{\prime})\right)  \ \text{for}  \ n\geq -1
$$
as follows.
We put $s_{-1}=0$. For $n\geq 0$ and a generator  $\phi\in \Omega_n^c\left(S^{\Box}_{I_k}(Q)\right)$ given by the singular cube $\phi\colon I^n_k\to Q\, (n\geq 0)$ we set
$s_{n}(\phi)\in {\Omega}_{n+1}^c\left(S^{\Box}_{I_k}(Q^{\prime})\right)$ be   the singular 
$(n+1)$-cube  that is given by the  composition of  maps 
\begin{equation*}
I^{n+1}_k=I_k\Box I_k^n\overset{\operatorname{Id}\Box \phi}{\longrightarrow } I_k\Box Q\overset{ \pi\Box \operatorname{Id}}{\longrightarrow } I\Box Q
\overset{F}{\longrightarrow}Q^{\prime}.
\end{equation*}
The
homomorphisms $s_n$ satisfy the properties 
\begin{equation*}
s_{n-1}\partial_n +\partial_{n+1}s_n=f_n-g_n
\end{equation*}
and hence give a chain homotopy between $f_*$ and $g_*$.
\end{proof}

For any quiver  $Q$, let $\mathcal{F}\colon \bold I\to \bold{CW}$ be the 
composition  
$$
 \bold I\overset{\mathcal I}{\longrightarrow}
 \bold{Cub}\overset{|\ |_{\bold{Top}}}{\longrightarrow }\bold{CW} \ \ \text{with}\ \ \mathcal{F}(I_k)=\left|S_{I_k}^{\Box}(Q)\right|_{\bold{Top}}
$$
and let $\mathcal{D}\colon \bold I\to \bold{Ch}$ be the composition
$$
 \bold I\overset{\mathcal I}{\longrightarrow}
 \bold{Cub}\overset{\mathcal C}{\longrightarrow }\bold{Ch}\ \ \text{with}\ \
\mathcal{D}(I_k)=\mathcal C\left(S_{I_k}^{\Box}(Q)\right)
$$
 which are  contravariant functors.

\begin{proposition} \label{p5.6} For any quiver  $Q$ the maps $p$ and $i$ induce the following maps:

i)  cell maps of $CW$-complexes 
$$
\mathcal F(p)\colon \left|S^{\Box}_{\bold{D}}(Q)\right|_{\bold{Top}}\to \left|S^{\Box}_{I_k}(Q)\right|_{\bold{Top}} \ 
\text{and} \ 
\mathcal F(i)\colon \left|S^{\Box}_{I_k}(Q)\right|_{\bold{Top}}\to \left|S^{\Box}_{\bold D}(Q)\right|_{\bold{Top}}
$$
such that $\mathcal F(i)\mathcal F(p)=\operatorname{Id}$.

ii)  chain  maps of chain complexes 
$$
\mathcal D(p)\colon \mathcal C\left(S^{\Box}_{\bold{D}}(Q)\right)\to  \mathcal C\left(S^{\Box}_{I_k}(Q)\right)\ 
\text{and} \ \
\mathcal D(i)\colon \mathcal C\left(S^{\Box}_{I_k}(Q)\right)\to \mathcal C\left(S^{\Box}_{\bold D}(Q)\right)
$$
such that $\mathcal D(i)\mathcal D(p)=\operatorname{Id}$.
\end{proposition}
\begin{proof} Follows from  Proposition \ref{p4.9}, definition of the functor $\mathcal C$, and Proposition \ref{p2.5}. 
\end{proof}

\begin{theorem}\label{t5.7} For any quiver $Q$, the quiver realization functors  $|\ |_{\bold Q}$ and $|\ |^{I_k}_{\bold Q}$ and the digraph maps $p$ and $i$ induce the  quiver maps
$$
i) \  |p^{\Box}|_{\bold Q}\colon  \left|S^{\Box}_{\bold{D}}(Q)\right|_{\bold{Q}}\to \left|S^{\Box}_{I_k}(Q)\right|_{\bold{Q}},  \ \
 \ 
|i^{\Box}|_{\bold Q}\colon \left|S^{\Box}_{I_k}(Q)\right|_{\bold{Q}}\to \left|S^{\Box}_{\bold D}(Q)\right|_{\bold{Q}}
$$ 
such that  $|i^{\Box}|_{\bold Q}\circ |p^{\Box}|_{\bold Q}=\operatorname{Id}$, and 
$$
ii) \ |p^{\Box}|_{\bold Q}^{I_k}\colon  \left|S^{\Box}_{\bold{D}}
(Q)\right|_{\bold{Q}}^{I_k}\to \left|S^{\Box}_{I_k}(Q)\right|_{\bold{Q}}^{I_k}, \ 
|i^{\Box}|_{\bold Q}^{I_k}\colon \left|S^{\Box}_{I_k}(Q)\right|_{\bold{Q}}^{I_k}\to \left|S^{\Box}_{\bold D}(Q)\right|_{\bold{Q}}^{I_k}
$$ 
such that  $|i^{\Box}|_{\bold Q}\circ |p^{\Box}|_{\bold Q}=\operatorname{Id}$.
\end{theorem}
\begin{proof}  Follows from  Propositions \ref{p4.9} and  \ref{p4.7}. 
\end{proof}

Now using the quiver realization of a cubical set we  introduce a collection of 
path homology theories on the category of  cubical sets and describe  their basic properties.  At first we recall the definition 
of path homology of a digraph  \cite{Axioms}, \cite{MiHomotopy}, \cite{Mi3}. 

Let $G=(V_G,E_G)$ be a digraph,  and $R$ be a commutative ring. 
An \emph{elementary path} on the set $V_G$ is defined as any sequence $i_{0},\dots,i_{p}$ of vertices and is denoted $e_{i_{0}\dots i_{p}}$. The elements of   a free $R$-module $\Lambda_p=\Lambda _{p}( V)$ which is generated by all elementary paths with fixed  $p\geq 0$  are called $p$-\emph{paths}.  We set $ \Lambda _{-1}=0$.  Define the \emph{boundary} operator $\partial\colon \Lambda _{p}\rightarrow \Lambda _{p-1}$ setting 
$
 \partial e_{i_{0}...i_{p}}=\sum\limits_{q=0}^{p}\left( -1\right)^{q}e_{i_{0}\dots \widehat{i_{q}}\dots i_{p}}$  for  $p\geq 1$ and $\partial =0$ for $ p=0$. Then 
 $\partial^2=0$ and we have a chain complex $ \Lambda_*=\Lambda _*(V)$.  For $p\geq 1$,  let  $I_{p}=I_p(V)\subset \Lambda _{p}$ be a submodule that is generated  by all elementary paths for which there are at least two equal consecutive vertices, and let  $I_0=I_{-1}=\{0\}$. Then  $\partial (I_{p})\subset I_{p-1}$ and we  define the chain complex  $\mathcal {R}_{*}$ setting 
$\mathcal {R}_{p}=\Lambda _{p}/I_{p}
$
with  an  indced differential  which we continue to denote by $\partial$.  The elements of  the module  $\mathcal{R}_p$ are called  \emph{regular  paths}.   
 For $p\geq 0$, a regular elementary path  $e_{i_{0}\dots i_{p}}$ is called 
\emph{allowed} if $(i_{k}\rightarrow i_{k+1})\in E_G$  for  $0\leq k\leq  p-1$. In particular, 
any  path $e_{i_0} (i_0\in V_G)$ is allowed. 
For $p\geq 0$, let $\mathcal A_{p}=\mathcal A_{p}(G,R)$ be  the  submodule  of $\mathcal R_{p}(V,R)$ that is generated 
by all the allowed elementary $p$-paths and set $\mathcal A_{-1}=0$.
Define a submodule $\Omega _{p}\subset \mathcal A_p$ setting 
\begin{equation}\label{5.2}
\Omega _{p}=\left\{ v\in {\mathcal A}
_{p}:\partial v\in {\mathcal A}_{p-1}\right\}.  
\end{equation}
 It is easy  to see that we obtain a chain complex 
$\Omega_*=\Omega_*(G,R)$. 

\begin{definition}\label{d5.8} \rm The homologies of   the chain complex $\Omega_*$ are  called \emph{path homologies of the digraph $G$} and denoted by   
$
H_p^{path}(G)=H_{p}\left(\Omega_*\right)$ for $p\geq 0$.
\end{definition}

For a  digraph map  $f\colon G\rightarrow H $  and for every $p\geq 0$ the induced map $
f_{*}\colon \Lambda _{p}(V_G)\rightarrow \Lambda _{p}(V_H)
$
is given on the basic elements by 
 the rule 
$
f_{\ast }\left( e_{i_{0}...i_{p}}\right)
=e_{f(i_{0})...f(i_{p})}.
$
  The map $f_{\ast }$ is a
morphism of chain complexes and   an  induced chain map of quotient chain complexes 
$
f_{\ast }:\mathcal{R}_{*}( V_G) \rightarrow \mathcal{R}_{*}(
V_H) 
$
 is well defined  on basic elements by the rule 
\begin{equation}\label{5.3}
f_{\ast }\left( e_{i_{0}...i_{p}}\right) =
\begin{cases}
e_{f(i_{0})...f(i_{p})}, & \text{if }e_{f(i_{0})...f(i_{p})}\text{ is
regular,} \\ 
0, & \text{if }e_{f(i_{0})...f(i_{p})}\ \text{otherwise.}%
\end{cases}
\end{equation}

It follows from (\ref{5.3}) and (\ref{5.2})   that we have an induced morphism of chain 
complexes $f_*\colon \Omega_*(G)\to \Omega_*(H)$ and an induced  homomorphism of homology groups $H_*^{path}(G)\to H_*^{path}(H)$. 
The path homology groups of digraphs is functorial and  homotopy invariant \cite{MiHomotopy}.

 Recall that by Proposition \ref{p4.15} the quiver realization of a simple cubical set is a digraph, and for  any cubical set $K$ and any line digraph $I_k$ with $k\geq 2$ the $I_k$ quiver realization $|K|_{\bold Q}^{I_k}$ is a digraph. 

\begin{definition}\label{d5.9} \rm  a)  \emph{The simple path homologies}  $H_p^{\bold Q}(K)$ of a  simple cubical set $K$ are defined as the path homologies of the quiver  realization $|K|_{\bold Q}$ that is 
$$
H_p^{\bold Q}(K)=H_p^{path}\left(|K|_{\bold Q}\right).
$$

b) Let $I_k$ be a line digraph with $k\geq 2$.  \emph{The $I_k$-path homologies}  $H_p^{I_k}(K)$ of a cubical set $K$ are defined as the path homologies of the quiver  realization $|K|_{\bold Q}^{I_k}$ that is 
$$
H_p^{I_k}(K)=H_p^{path}\left(|K|_{\bold Q}^{I_k}\right).
$$
\end{definition}

Now we consider an example which illustrate the differences between various homology groups for simple cubical sets.  

\begin{example}\label{e5.10}\rm i) Let $K$ be the simple cubical set with the nondegenerate cubes $K_0=\{\bold 0,\bold 1,\bold 2\}$, $K_1=\{\bold i, \bold j, \bold k\}$ and the face maps $\partial^0_1(\bold i)=\bold 0, \partial^1_1(\bold i)=\bold 1, 
\partial^0_1(\bold j)=\bold 1, \partial^1_1(\bold j)=\bold 2, \partial^0_1(\bold k)=\bold 0, \partial^1_1(\bold k)=\bold 2$. Then by 
Definition \ref{d5.1} and Proposition \ref{p5.2} we have 
$$
H_n(K)=\begin{cases} R& \text{for} \ n=0,1,\\
                                 0 & \text{for} \ n\geq 2.\\
 \end{cases}
$$
By   Definition \ref{d4.5}  the simple quiver realization $|K|_{\bold Q}$ gives the digraph $G$
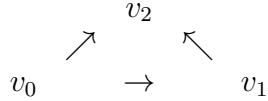
$$
\begin{matrix}
&&v_2&&\\
&\nearrow &   &\nwarrow&\\
v_0 &&\rightarrow&&v_1\\
\end{matrix}
$$
with the vertices $v_0,v_1,v_2$ which correspond to the zero-dimensional cubes $\bold 0,\bold 1,\bold 2$, respectively.
Hence by the definition of path homologies  the nontrivial modules $\Omega_p$ defined in (\ref{5.2}) are generated as follows 
$$
\Omega_0=\langle e_{v_0},   e_{v_1},  e_{v_2}\rangle, \  \Omega_1=\langle e_{v_0v_1},   e_{v_1v_2},  e_{v_0v_2}\rangle, \
\Omega_2=\langle e_{v_0v_1v_2}\rangle, 
$$
and we have the following nontrivial differentials which are given  on the basic elements by $\partial e_{v_0v_1v_2}=  e_{v_1v_2}-e_{v_0v_2}+
e_{v_0v_1}$, $\partial e_{v_0v_1}=e_{v_1}-e_{v_0}$, $\partial e_{v_0v_2}=e_{v_2}-e_{v_0}$, $\partial e_{v_1v_2}=e_{v_2}-e_{v_1}$. Hence, by Definition \ref{d5.9}, 
$$
H_n^{\bold Q}(K)=\begin{cases} R& \text{for} \ n=0,\\
                                 0 & \text{for} \ n\geq 1.\\
 \end{cases}
$$
Similarly to the case $H_*^{\bold Q}$ we can consider the case of homology groups $H_*^{I_k}(K)$ for $k\geq 1$. For every $k\geq 1$ we obtain 
$$
H_n^{I_k}(K)=\begin{cases} R& \text{for} \ n=0,1,\\
                                 0 & \text{for} \ n\geq 2.\\
 \end{cases}
$$

 ii) Let $K^{\prime}$ be the simple cubical set with the non-degenerate cubes 
$K_0^{\prime}=\{\bold 0,\bold 1,\bold 2\}$, $K_1^{\prime}=\{\bold i, \bold j, \bold k\}$ and the face maps $\partial^0_1(\bold i)=\bold 0, \partial^1_1(\bold i)=\bold 1, 
\partial^0_1(\bold j)=\bold 1, \partial^1_1(\bold j)=\bold 2, \partial^0_1(\bold k)=\bold 2, \partial^1_1(\bold k)=\bold 0$. Similarly to above  we have 
$$
H_n(K^{\prime})=\begin{cases} R& \text{for} \ n=0,1,\\
                                 0 & \text{for} \ n\geq 2.\\
 \end{cases}
$$
In this case  the simple quiver realization $|K^{\prime}|_{\bold Q}$ gives the digraph $G^{\prime}$
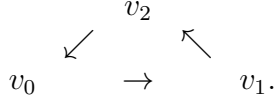
$$
\begin{matrix}
&&v_2&&\\
&\swarrow &   &\nwarrow&\\
v_0 &&\rightarrow&&v_1.\\
\end{matrix}
$$
Furthermore,  the consideration is the similar to the case i) of the Example \ref{e5.10},  and we obtain 
$$
H_n^{\bold Q}(K^{\prime})=\begin{cases} R& \text{for} \ n=0,1\\
                                 0 & \text{for} \ n\geq 2\\
 \end{cases}
$$
since in this case $\Omega_2(G^{\prime})=0$.  For every $k\geq 1$ we obtain 
$$
H_n^{I_k}(K^{\prime})=\begin{cases} R& \text{for} \ n=0,1\\
                                 0 & \text{for} \ n\geq 2.\\
 \end{cases}
$$
\end{example}

\begin{example} \label{e5.11} \rm Let $S$ be a finite simplicial complex. Define the  digraph $G_S$ with the set of vertices 
$V_{G_S}$ consisting of the simplexes of $S$,   and there is an arrow 
$(\sigma \to \tau) \in E_{G_S}$ for $\sigma, \tau \in V_{G_S}$ if and only if  $\sigma\supset \tau$ and $\operatorname{dim} \sigma = \operatorname{dim} \tau +1$ (see  \cite{Mi3}). 
 Then  as follows from the definition,   the digraph $G_S$ can be considered  as the quiver  realization $|K_S|_{\bold Q}$ of the simple  cubical set $K_S$ in which non-degenerate cubes correspond to cubes of the digraph $G_S$.  Moreover, the $CW$-complex  $|K_S|_{\bold{Top}}$ is homeomorphic to the polyhedron $|S|$, and there are 
the following isomorphisms of homology groups
$$
H_*(|S|)=H_*(|K_S|_{\bold{Top}})=H_*^{path}(G_S)=H_*^{path}(K_S). 
$$
 
\end{example}

From now we return to the consideration of an arbitrary (no necessary simple) cubical  set $K$
and  its quiver realizations  $|K|_{\bold Q}$ and $|K|_{\bold Q}^{I_k}$ which were  defined in Definitions \ref{d4.5}, \ref{d4.6}.
Now we  can define the path homology of the  cubical set $K$ using the path homology which was constructed in  \cite{Forum} for arbitrary  finite quiver.

\begin{definition}
\label{d5.12} \rm  For $n\geq 1$, an \emph{\ elementary $n$-path $p$ in a quiver} $Q=(V,E,s,t)$ is a non-empty sequence $a_{1},\dots
,a_{n}$ of arrows in $Q$ such that $t(a_{i})=s(a_{i+1})$ for $1\leq i\leq n-1$.  Denote this path by $p=a_{1}\dots a_{n}$. Define \emph{\ the
start vertex of }$p$ by $s(p)=s\left( a_{1}\right) $ and \emph{the end
vertex of }$p$ by $t(p)=t\left( a_{n}\right) $.

For $n=0$ \emph{an elementary $0$-path} is defined by  $p:=v$
where $v\in V$ is any vertex and  we define $s(p)=t(p)=v$. The number $n$ is called the length of a $n$-path $p$ and is denoted by $
\left\vert p\right\vert$.  The set of all elementary $n$-paths in  $Q$ is denoted by $P_{n}Q$ and $PQ=\bigcup_{n\geq 0}P_{n}Q$.  
\end{definition}

Let  $\Lambda_{*}(Q) =\sum_{n\geq 0}\Lambda_n(Q)$ be the graded  $R$-module 
where $\Lambda_n(Q)$  is the free $R$-module spanned by all  elementary
$n$-paths in $Q$.  We set  $\Lambda_{-1}(Q)=0$. We can equip  $\Lambda_{*}(Q)$ by a structure of a graded algebra. We define a multiplication on the set of elementary paths using \emph{operation of concatenation} and  extend this operation by linearity. For elementary paths  $p=a_{1}\dots
a_{n}$ and $q=b_{1}\dots b_{m}$ with $n,m\geq 1$ the concatenation  $p\cdot q\in PQ$ is defined by setting 
$p\cdot q=
a_{1}\dots a_{n}b_{2}\dots b_{m}$ if $t(a_{n})=s(b_{1})$,   and $p\cdot q=0$
otherwise. For the paths $p=v\in V$ and $q=b_{1}\dots b_{m}$, we set 
$p\cdot q= q$ if $v=s(b_{1})$, $q\cdot p =q$ if $v=t(b_m)$,   and $p\cdot q=0, q\cdot p=0$ otherwise.  For the paths $p=v$, $q=w$ where $v,w\in V$, we set
$p\cdot q=v$ if $v=w$, and  $p\cdot q=0$ otherwise. From now we suppose  that no positive integer in the commutative ring $R$ is a zero divisor.  

\begin{definition}
\label{d5.13} \rm A quiver $Q$ is called \emph{complete of power} $N$ if, for
any two vertices $v,w$ there is exactly $N$ arrows with the start vertex $v$
and the end vertex $w$.
\end{definition}

 Let $Q=(V,E,s,t)$ be a complete quiver of  power $N\geq 1$. Define a
product 
$
\Lambda _{1}(Q)\times \Lambda _{1}(Q)\rightarrow \Lambda _{1}(Q)
$
on the arrows $a,b\in E$ by setting 
$[ab]=\sum c$ for $t(a)=s(b),s(c)=s(a),t(c)=t(b)$ and $[ab]=
0$ otherwise,  and then extend it by linearity.   For $n\geq 0$ and $0\leq i\leq n+1$, define  homomorphisms 
$
\partial _{i}\colon \Lambda _{n+1}(Q)\rightarrow \Lambda _{n}(Q)
$ on the basic elements as follows: $\partial_0 v=0$ for $v\in V$, 
$\partial _{0}a=Nt(a)$ and $\partial _{1}a=Ns(a)$ for $a\in E$,  
$\partial _{0}(a_1\dots a_n)=N(a_{2}\dots a_{n})$  and $\partial
_{n}(a_1\dots a_n)=N(a_{1}\dots a_{n-1})$  for $n\geq 2$, 
$
\partial _{i}(a_0\dots a_n)=a_{0}\dots a_{i-2}[a_{i-1}a_{i}]a_{i+1}\dots
a_{n}$
for $n\geq 1$, $1\leq i\leq n$.

\begin{theorem}\label{t5.14} {\rm \cite{Forum}}  Let $Q$ be a complete quiver of power $N$. For all $n\geq -1$, define
homomorphisms $\partial \colon \Lambda _{n+1}(Q)\rightarrow \Lambda _{n}(Q)$
by 
$ \partial =\sum_{i=0}^{n+1}(-1)^{i}\partial _{i}$.
Then $\partial ^{2}=0$ and $\Lambda_*(Q)$ is  a  chain complex with the differential $\partial$. 
\end{theorem}

For any quiver $Q=(V,E,s,t)$ we can define its completion quiver $\widetilde{Q}=(\widetilde{V},\widetilde{E},%
\widetilde{s},\widetilde{t})$ as follows.  Let $\mu(v,w)$ be a number of arrows with the   start vertex $v$ and the end vertex $w$,  and let $M\geq \operatorname{max} \{\mu(v,w)|\, v,w\in V\}$.  We put $\tilde{V}=V$ and $\widetilde{E}$ is obtained from $E$   by adding  $M-\mu
(v,w)$ new arrows from $v$ to $w$ for any  ordered pair $v, w\in V$.   Thus we obtain 
 a complete quiver  $\widetilde{Q}$ of power $M$ and  a natural inclusion of quivers $\tau \colon Q\rightarrow \widetilde{Q}$ which  induces an inclusion of $R$-modules 
$
\tau _{\ast }:\Lambda _{n}(Q){\rightarrow }\Lambda _{n}(\widetilde{Q})$
for $n\geq 0$. We have the restriction $\partial|_{\Lambda _{n+1}(Q)}\colon \Lambda _{n+1}(Q)\to \Lambda _{n}(\widetilde Q)$ of the differential $\partial \colon \Lambda _{n+1}(\widetilde Q)\rightarrow \Lambda _{n}(\widetilde Q)$ which we continue to denote by $\partial$.  For $n\geq 0$, consider the
submodules of $\Lambda _{n}(Q)$ given by 
\begin{equation*}
\Omega _{n}^M(Q):=\left\{ v\in \Lambda _{n}(Q):\partial v\in \Lambda
_{n-1}(Q)\right\}. 
\end{equation*}
We have 
$
\partial \left( \Omega _{n}^M\left( Q\right) \right) \subset \Omega
_{n-1}^M\left( Q\right)$ as  follows directly from the identity $\partial ^{2}=0$ in $\Lambda
_{\ast }(\widetilde{Q})$. Hence, $\Omega _{\ast }^M\left( Q\right) $  with the induced from 
$\partial$ differential is a chain complex.

\begin{definition} 
\label{d5.15} \rm  Let $Q$ be a quiver and  let $M\geq \operatorname{max} \{\mu(v,w)|\, v,w\in V\}$.  For  $n\geq 0$ the 
\emph{M-homology groups of the quiver} $Q$ \emph{with coefficients in} $R$   are defined as follows
$
H_{n}^M(Q,R)=H_{n}(\Omega _{\ast }^M(Q))
$, 
\end{definition}

Let $K$ be a cubical set.  For every ordered pair $(v, w)$  of $0$-dimensional cubes (including the case $(v,v)$),  we denote by $\mu(v, w)$  the number of $1$-dimensional cubes $a\in K_1$ such that $\partial_1^0(a)=v, \partial_1^1(a)=w$. Let $M\geq \operatorname{max} \{\mu(v,w)|\, v,w\in K_0\}$. Then the quiver 
$|K|_{\bold Q}=(V,E,s,t)$ is the sub-quiver of the complete   quiver $\widetilde{Q}=(\widetilde{V},\widetilde{E},
\widetilde{s},\widetilde{t})$ of power $M$  that was defined above.

\begin{definition} 
\label{d5.16} \rm  The \emph{simple quiver M-homology groups}  of  a cubical set $K$$  $\emph{\ with coefficients in }$R$   are defined by 
$
H_{n}^{M}(K,R):=H_{n}\left(\Omega _{\ast}^M(|K|_{\bold Q})\right)
$. 
\end{definition}

Let $K$ be a connected cubical set. We note that the  modules $\Omega _{n}^M(|K|_{\bold Q})$ and  $\Omega _{n}^{M^{\prime}})(|K|_{\bold Q})$ are isomorphic for any numbers $M, M^{\prime}\geq 
\operatorname{max} \{\mu(v,w)|\, v,w\in K_0\}$  and $n\geq 0$  but the differentials of the corresponding chain complexes  are different \cite[Th. 4.4]{Forum}. 
 
\begin{example}\label{e5.17} \rm Let $K$ be a cubical set  which has only  the nondegenerated cubes $K_0=\{v_0,v_1,v_2\}, K_1=\{a_1,a_2,b_1,b_2\}$  and $\partial_1^0(a_1)=\partial_1^0(a_2)=v_0$, $\partial_1^1(a_1)=\partial_1^1(a_2)=v_1$, 
$\partial_1^0(b_1)=\partial_1^0(b_2)=v_1$, $\partial_1^1(b_1)=\partial_1^1(b_2)=v_2$. 
We have the following quiver 
realization  $Q=|K|_{\bold Q}$:
\begin{equation*}
v_0\overset{\overset{a_1}\longrightarrow}{\underset{a_2}\longrightarrow}  v_1\overset{\overset{b_1}\longrightarrow}{\underset{b_2}\longrightarrow}  v_2.
\end{equation*}

 For the quiver $Q$ we have 
$\operatorname{max}\{\mu(v,w)|v,w\in V\}=2$. Let $M=2$ and we compute homologies 
$
H_{n}^2(K,\mathbb R)=H_{n}(\Omega _{\ast }^2(Q))
$ 
with   coefficient the real numbers.

We have 
$
\Lambda_{0}(Q)=\langle v_{0},v_{1},v_{2}\rangle$,
$$
\Lambda_1(Q)=\langle a_1,a_2,b_1,b_2\rangle,  \ \    
\Lambda_2(Q)=\langle a_1b_1, a_1b_2, a_2b_1, a_2b_2\rangle, 
$$
and $\Lambda_{i}(Q)=0$  for $i\geq 3$. Direct computation  gives 
$$
\Omega _{0}^{2}(Q)=\langle v_{0},v_{1},v_{2}\rangle ,\ \ \Omega
_{1}^{2}(Q)=\langle a_1,a_2, b_1,b_2\rangle,\
$$
$$
 \Omega_{2}^{2}(Q)=\langle a_1b_1 -a_1b_2, a_1b_1 -a_2b_1, a_1b_1 -a_2b_2, 
a_2b_1 -a_2b_2,
 a_2b_1 -a_1b_2, 
a_2b_2 -a_1b_2.
\rangle
$$
and $\Omega _{i}^{2}(Q)=0$ for $i\geq 3$. There are only three independent elements in $\Omega_{2}^{2}(Q)$, and we obtain  $
 \Omega_{2}^{2}(Q)=\langle a_1b_1 -a_1b_2, a_1b_1 -a_2b_1, a_1b_1 -a_2b_2\rangle$. 
Thus $\operatorname{rank} \Omega _{0}^{2}(Q)=3, \operatorname{rank} \Omega _{1}^{2}(Q)=4,  \operatorname{rank} \Omega _{2}^{2}(Q)=3$. Direct computation  of the  differential  gives that the rank of the image of $\partial_1$  in $\Omega_0^2(Q)$ is 2 
and, hence $H_0^2(K,\mathbb R)=\mathbb R$. It follows from this  that 
$$
\operatorname{rank} \left[\operatorname{Ker}\{\partial_1\colon \Omega _{1}^{2}(Q)\to 
\Omega _{0}^{2}(Q)\}\right]=2.
$$
The image of the differential $\partial_2\colon \Omega _{2}^{2}(Q)\to 
\Omega _{1}^{2}(Q)$ is generated by the elements $a_1-a_2, b_1-b_2$. Hence 
$$
\operatorname{rank} \left[\operatorname{Im}\{\partial_2\colon \Omega _{2}^{2}(Q)\to 
\Omega _{1}^{2}(Q)\}\right]=2.
$$
Thus, we obtain $H_1^2(K,\mathbb R)=0$, $H_2^2(K,\mathbb R)=\mathbb R$. 
\end{example}

\end{document}